\documentclass[reqno,11pt,centertags]{article}
\usepackage{amsmath,amsthm,amscd,amssymb,latexsym,upref}
\date{\today}
\input epsf
\usepackage{epsfig}
\usepackage[T2A,OT1]{fontenc}
\usepackage[ot2enc]{inputenc}
\usepackage[russian,english]{babel}

\usepackage{epic,eepicemu}

\newcommand{\bbD}{{\mathbb{D}}}
\newcommand{\bbE}{{\mathbb{E}}}

\newcommand{\bbR}{{\mathbb{R}}}

\newcommand{\bbC}{{\mathbb{C}}}

\newcommand{\bbT}{{\mathbb{T}}}

\newcommand{\cP}{{\mathcal{P}}}

\newcommand{\cR}{{\mathcal{R}}}

\newcommand{\cE}{{\mathcal{E}}}

\newcommand{\cM}{{\mathcal{M}}}
\newcommand{\cB}{{\mathcal{B}}}
\newcommand{\fB}{{\mathfrak{B}}}

\newcommand{\e}{{\epsilon}}
\renewcommand{\l}{{\ell}}
\newcommand{\om}{{\omega}}

\newcommand{\fA}{{\mathfrak{A}}}

\newcommand{\z}{\zeta}
\renewcommand{\Re}{\text{\rm Re\,}}
\renewcommand{\Im}{\text{\rm Im\,}}

\newcommand{\sgn}{\text{\rm sgn}}

\allowdisplaybreaks \numberwithin{equation}{section}
\newtheorem{theorem}{Theorem}[section]

\newtheorem{lemma}[theorem]{Lemma}

\newtheorem{corollary}[theorem]{Corollary}

\theoremstyle{definition}

\newtheorem{remark}[theorem]{Remark}
\newtheorem{problem}[theorem]{Problem}

\title{On $L^1$ extremal problem for entire functions }
\author{ P. Yuditskii\thanks{Supported by the Austrian Science Fund FWF, project no: P22025-N18.}}

\date{\today}
\begin{document}
\maketitle

\begin{abstract}
We generalized the Korkin-Zolotarev theorem to the case of entire functions having the smallest $L^1$ norm on a system of intervals $E$. If $\bbC\setminus E$ is a domain of Widom type with the Direct Cauchy Theorem we give an explicit formula for the minimal deviation. Important relations between the problem and the theory of canonical systems with reflectionless resolvent functions are shown.

\end{abstract}

\section{Introduction}

The classical $L^2-, L^1-,$ and $L^\infty$-extremal problems for polynomials and entire functions have numerous connections with diverse problem in analysis: the investigations of Abel on expressing elliptic and hyperelliptic integrals in elementary functions; continued fractions and orthogonal polynomials; spectral properties of periodic (and almost periodic) Jacobi matrices and Schr\"odinger operators; factorization of functions on Riemann surfaces; subharmonic majorants; special conformal mappings onto ``comb-like domains"; Pell's equations and so on.

 Probably the earliest version of the results we are interested in is the following Korkin-Zolotarev theorem (although implicitly it was already in \cite{TCH}).
 \begin{theorem} Let $\cP_n$ be the set of polynomials of the degree at most $n$. Then
 \begin{equation}\label{eq1}
  M:=  \inf_{P(x)=x^n+...\in\cP_n}\int_{-1}^1|P(x)|dx=\frac 1{2^{n-1}},
 \end{equation}
 and the extremal polynomial $U_n(x)$ is the Chebychev polynomial of the second kind, i.e.,
 \begin{equation*}
    U_n(x)=\frac{1/\z^{n+1}-\z^{n+1}}{2^n(1/\z-\z)},\quad x=\frac 1 2(1/\z+\z).
\end{equation*}
\end{theorem}

Various generalizations were given by Stieltjes, Markov, Posse, Bernstein, Akhiezer and Krein,... The following two directions are the most important for us: we would like to pass from the polynomial case to classes of entire functions and from a single interval of integration  to a system of intervals, ideally to an arbitrary closed subset of $\bbR$.

A two interval version of the Korkin--Zolotarev problem was  investigated by N.I. Akhiezer \cite{AKH36}. 
It does not seem obvious that the famous Akhiezer's polynomials orthogonal on two intervals were 
first constructed in this paper (at least its \textit{r\'esum\'e}  does not contain any hint in this direction).
The several interval case was studied by Akhiezer and Krein in the following general setting \cite[Chapter VIII, Sect. 9]{KN}. 
\begin{problem}\label{prpextr} Let $E$ be a system of intervals, $E=[b_0,a_0]\setminus \cup_{j=1}^m(a_j,b_j)$. For a given real vector $\{\lambda_k\}_{k=0}^n$, $\sum \lambda_k^2>0$, define the functional on $\cP_n$,
\begin{equation*}
    \Lambda(P)=\sum \lambda_k c_k,\quad P(x)=\sum c_k x^k.
\end{equation*}
Find
\begin{equation*}\label{eq2}
   M=M(\Lambda)= \inf_{P\in\cP_n, \Lambda(P)=1}\int_E|P(x)|dx.
 \end{equation*}
\end{problem}
Of course very interesting partial cases of the problem correspond to the choices $\Lambda(P)=\Lambda_{x_0}(P)=P(x_0)$, $x_0\in \bbR\setminus E$, and $\Lambda_\infty$ when $\lambda_k=0$ for every $0\le k<n$ and $\lambda_n=1$, see \eqref{eq1}.

Due to the duality principle \cite{AK38} this problem is equivalent to the so called Markov's $L$ moment problem,
namely, $M(\Lambda)$ is equal to $1/L(\Lambda)$ in the following problem.
\begin{problem}\label{PLM} Find the smallest $L=L(\Lambda)>0$ such that the moment problem
\begin{equation*}\label{eq3}
    \int_E x^k f(x)dx=\lambda_k
\end{equation*}
is solvable with a real function $f(x)$ such that $|f(x)|\le L$, $x\in E$.
\end{problem}

Problem \ref{PLM} was investigated in \cite{AKH49}. One has to consider
$2^m$ sequences $\{s_k\}_{k=0}^{n+1}=\{s_k(\delta_1,...\delta_m)\}_{k=0}^{n+1}$, $\delta_j=\pm 1$ given by the expansions
\begin{equation}\label{eq4}
\frac{s_0}z+\frac{s_1}{z^2}+...+\frac{s_{n+1}}{z^{n+2}}+...=\sqrt{\frac{z-a_0}{z-b_0}}\prod_{j\ge 1}\sqrt{\frac{z-a_j}{z-b_j}}^{\delta_j}e^{\frac 1{2L}\left(\frac{\lambda_0}z+\frac{\lambda_1}{z^2}+...+\frac{\lambda_n}{z^{n+1}}+...\right)}
\end{equation}
Then $L$ is the smallest value for which all $\{s_k\}_{k=0}^{n+1}$'s are moments of positive measures supported
on $[b_0,a_0]$. There are  well known algebraic conditions for a sequence of numbers to be moments of a positive measure given in terms of positivity of the associated Hankel matrices, so one gets a system of algebraic relations  for $L$.

As it was found by F. Peherstorfer \cite{P88}  in the case of the functional $\Lambda_\infty$ the solution can be represented even explicitly in terms of polynomials orthogonal with respect to the Stieltjes function given by RHS of \eqref{eq4}. Indeed, the solution can be given by the orthogonal polynomial which has minimal $L^2$-norm among all possible choices $\delta_j\in\{\pm 1\}$. It is that (unique) orthogonal polynomial which has no zeros in the open gaps $(a_j,b_j)$ (all zeros in $E$). 
A connection between the $L^1$-extremal problem and orthogonal polynomials in the one interval  case  for a weighted problem was noted in \cite{AK61}.

There is an enormous amount of literature dealing with ``continuous analogues of orthogonal polynomials on a system of intervals" \cite{AKH61}, that is, $L^2$ extremal problems. In fact, such orthogonalization represents the key element in the spectral theory for periodic and almost periodic differential operators, particularly, Sturm-Liouville operators \cite{MAR}. In their turns,  they are especially important in the Theory of Integrable Systems, particularly, KdV, see e.g. \cite{MUM2}.

Concerning $L^\infty$ extremal problems  on closed subsets of $\bbR$ see e.g. the review \cite{SYU93}.
Going to a rigorous formulation of the $L^1$ extremal problem, first of all, we have to point out that the class of entire functions, in terms of their growth at infinity, depends essentially on the set $E$. For instance, in  the Sturm-Liouville theory, where the ``spectral set" $E$ is basically a half axis, it is required that $F(z^2)$ is an entire function of exponential type $\l$: a typical function is $\cos \sqrt{z} \l$.

Let $E$ be a system of intervals, which accumulate only at $-\infty$,
\begin{equation}\label{eq101.5}
    E=(-\infty,-1)\setminus \cup_{j\ge 1}(a_j,b_j), \quad a_j\to-\infty,  \ \text{as}\ j\to\infty.
\end{equation}
Let $G(z,z_0)$ be the Green function of the domain $\Omega=\bbC\setminus E$ with a logarithmic pole at $z_0\in\Omega$. In what follows we assume that $\Omega$ is of Widom type \cite{HA}, i.e.,
\begin{equation}\label{widomcond}
\sum_{c_j:\nabla G(c_j,0)=0} G(c_j,0)<\infty.
\end{equation}
The Martin function of the domain, see e.g. \cite{EY}, is defined as the following limit
$$
\cM(z)=\lim_{z_0\to+\infty}\frac{G(z,z_0)}{G(0,z_0)}
$$
Note that it meets the normalization $\cM(0)=1$.

%%%%%%%%%

 We say that an entire function $F(z)$ belongs to the class $\cB_E(\l)$ if
 \begin{itemize}
 \item[(i)]it is of bounded characteristic in the domain  $\Omega=\bbC\setminus E$,
 \item[(ii)] $\lim_{x\to\infty}\log|F(x)|/\cM(x)\le \l$.
 \end{itemize}

 We study the following extremal problem:
 \begin{problem}\label{pr1} For $\l\ge 0$ find
\begin{equation}\label{1}
    M(\l)=\inf\left\{\int_E |F(x)|\frac{dx}{|x|}: F\in \cB_E(\l), F(0)=1 \right\}.
\end{equation}
\end{problem}

Now we describe briefly its solution.
First, to the given set $E$ we associate the following family of Nevanlinna class functions (we say that $w(z)$ belongs to the Nevanlinna class if it is holomorphic in the upper half-plane and $\Im w(z)\ge 0$)
\begin{equation}\label{rth01}
  \cR=\left\{  R(z)=-\frac {\sqrt{1+z}} z\sqrt{\prod_{j\ge 1}\left(\frac{1-z/a_j}{1-z/b_j}\right)^{\delta_j}}:\delta_j=\pm 1\right\}.
\end{equation}
Note that for each of this function we have the following integral representation
\begin{equation}\label{rth02}
    R(z)=-\frac 1 z+q_0+
   \frac 1 \pi \int_E \left\{\frac 1{x-z}-\frac 1{x}\right\}{|R(x)|}\,dx.
\end{equation}

According to the famous de Branges Theorem \cite{dB} $R(z)$ is the resolvent function (respectively $|R(x)| dx$ is the spectral measure) of a certain canonical system.

\begin{theorem}\label{thdB} For the given $R(z)\in \cR$ there exists a unique (up to a monotonic change of the independent variable $t$) non-negative $2\times 2$ matrix $H(t)=H_R(t)$, $t\in[0,\infty)$, such that for all $t$ the function $R(z)$ possesses the representation
   \begin{equation}\label{repr0}
        R(z)=\frac{A(t,z)\cE(t,z)+B(t,z)}{C(t,z)\cE(t,z)+D(t,z)},
     \end{equation}
where $\cE(t,z)$ is a Nevanlinna class function and
\begin{equation}\label{debr0}
   \frac d{dt} \begin{bmatrix}A&B\\C&D
     \end{bmatrix}(t,z)\begin{bmatrix}0&-1\\1&0
     \end{bmatrix}=\begin{bmatrix} A&B\\C&D
     \end{bmatrix}(t,z)zH(t), \ \begin{bmatrix}A&B\\C&D
     \end{bmatrix}(0,z)=I.
\end{equation}
Moreover,
$$
R(z)=\lim_{t\to \infty}\frac{A(t,z)}{C(t,z)}=\lim_{t\to \infty}\frac{B(t,z)}{D(t,z)}
$$
\end{theorem}

Note that in a certain sense the matrix
\begin{equation}\label{et0}
    \begin{bmatrix} B_R'(t,0)&-A_R'(t,0)\\ D_R'(t,0)&-C'_R(t,0)
    \end{bmatrix}(t)=
   \int_0^t H_R(t)dt
\end{equation}
represents  the ``matrix exponential type" of the entire matrix function
$$
\fA(t,z)=  \begin{bmatrix}A&B\\C&D
     \end{bmatrix}(t,z).
     $$

     Now, we choose
     \begin{equation}\label{tt0}
        t_0=\sup\{t:\lim_{x\to\infty}\ln |(AC)(t,x)|/ \cM(x)\le \l\},
     \end{equation}
     (the ``exponential type" is monotonic with respect to $t$). In this case
     \begin{equation}\label{report0}
        M(\l)=\inf_{R\in\cR}\ln\frac 1{{C'_R(t_0,0)}^2}
     \end{equation}
     and
     the extremal function of Problem \ref{pr1} is of the form
     \begin{equation}\label{sol0}
        F(z)=\frac{(AC)(t_0,z)}{zC'(t_0,0)}.
     \end{equation}
     
     The main results of the paper are given in Theorems \ref{th55} and \ref{th61}. If in a Widom domain the Direct Cauchy Theorem (DCT) holds (for the definition see Section 5) we obtain a very much explicit formula \eqref{pimain} for $M(\l)$. Analizing $L^1$ extremal problem in Widom domains when DCT fails, we obtain an important characteristic property for canonical system with reflectionless resolvent functions, see Theorem \ref{th62}.

\section{Preliminaries}

Let $\cM_*(z)$ be the conjugated harmonic function to $\cM$. It is well defined in the upper half-plane, so that
\begin{equation}\label{conjmart}
\Theta(z)=-\cM_*(z)+i\cM(z)
\end{equation}
is of Nevanlinna class.

Let
\begin{equation}\label{uc}
    z:\bbD/\Gamma\to \Omega, \ z(0)=0, \  z'(0)>0,
\end{equation}
be the universal covering of $\Omega$. Here $\Gamma$ is a suitable Fuchsian group.
Then $e^{i\l\Theta(z(\zeta))}$ is a character automorphic function in $\bbD$, i.e., 
\begin{equation}\label{eq23}
e^{i\l\Theta}\circ\gamma=\xi_\l(\gamma) e^{i\l\Theta}, \ \gamma\in \Gamma,
\end{equation}
the system of multipliers $\xi_\l=\{\xi_\l(\gamma)\}_{\gamma\in\Gamma}$ belongs to the group of characters $\Gamma^*$ of the discrete group $\Gamma$.

Due to the Widom theorem the condition \eqref{widomcond} for an arbitrary character $\alpha\in\Gamma^*$ the Hardy class of holomorphic bounded character automorphic functions is not trivial, that is $H^\infty(\alpha)$ contains a non-constant function. Later we will show that it implies that $\cB_E(\l)$ is not trivial. Now
let us prove existence of the extremal function in this case.

Recall that in the Widom case $z'(\zeta)$ is  an outer function and the Harmonic measure on $E$ has absolutely continuous density, which is also the module of the outer function. Therefore, any function $F$
of $\cB_E(\l)$ is of the form
\begin{equation}\label{3}
F(z(\zeta))=e^{-i\l\Theta(z(\zeta))}Z(\zeta) f(\zeta),
\end{equation}
where $Z(\zeta)$ is an outer function in $\bbD$, depending on $E$ only, and $f$ is a function of Smirnov class in $\bbD$,  moreover
\begin{equation}\label{4}
   \int_E |F(x)|\frac{dx}{|x|}=\int_\bbT|f(\zeta)|dm(\zeta).
\end{equation}
That is \eqref{3} sets the correspondence $F\mapsto f$ from $\cB_E(\l)$ to $H^1$ in the unite disk.

Let $\{F_n\}$ be an extremal sequence for $M(\l)$. In the corresponding sequence
$\{f_n\}$ we chose a subsequence that converges uniformly on compact subsets in $\bbD$,
$f(\z)= \lim f_{n_k}(\z)$. In this case
$$
\int_{\bbT}|f(r\z)|dm(\z)=\lim\int_E | f_{n_k}(r\z)|dm(\z)\le \lim_{n\to\infty}\|F_n\|=M(\l).
$$
Thus $f$ belongs to $H^1$, and hence $F(z(\zeta))=e^{-i\l\Theta(z(\zeta))}Z(\zeta) f(\zeta)\in \cB_E(\l)$.
Therefore,
$$
M(\l)\le \int_E |F(x)|\frac{dx}{|x|}=\int_{\bbT}|f(\z)|dm(\z)\le M(\l).
$$

We note that the classes $\cB_E(\l)$ are evidently monotonic with respect to $E$, $\cB_E(\l)\subset\cB_{E_1}(\l)$ if $E\subset E_1$. Indeed, if $F$ is of bounded characteristic in $\bbC\setminus E$ it is of bounded characteristic in the smaller domain. Also for a small extension of $E$
$$
\cM_{E_1}(z)=\frac{\cM_E(z)-\int_{E_1\setminus E}\cM_E(x)\omega_{E_1}(dx,z)}{{1-\int_{E_1\setminus E}\cM_E(x)\omega_{E_1}(dx,0)}}.
$$
Therefore
$$
\lim_{x\to+\infty}\frac{\log|F(x)|}{\cM_{E_1}(x)}=\left(1-\int_{E_1\setminus E}\cM_E(x)\omega_{E_1}(dx,0)\right)
\lim_{x\to+\infty}\frac{\log|F(x)|}{\cM_{E}(x)}\le \l.
$$

Thus $\cB_{{(-\infty.-1]}}(\l)$ is the maximal element of the family. If $F\in\cB_{\bbR_-}(\l)$, the function $F(-z^2)$ is of bonded characteristic in the upper/lower half-plane. This class of functions is described in details  in \cite[Chapter V]{LE}. 
In what follows a certain maximal principle in this class will be permanently in use.
\begin{lemma}\label{l21}
Let $F$ be an entire function of bounded characteristic in ${\bbC\setminus \bbR_-}$. If $F$ is bounded on the positive half-axis $\bbR_+$ then $F$ is a constant function.
\end{lemma}

\begin{proof}
 By Theorem 4 \cite[Chapter V]{LE}, $F(-z^2)$ is of exponential type $\sigma$, moreover
\begin{equation}\label{typezero}
\sigma=\lim_{y\to+\infty}\frac{\log|F(-(iy)^2)|}{y}=0.
\end{equation}

Since $F(-(iy)^2)$, $y\in R$, is bounded, it is also of bounded characteristic in the left/right half-plane, see Theorem 11 in this Chapter. Therefore, having in mind \eqref{typezero}, we get
$$
\log |F(-z^2)|\le \frac 1 \pi\int_\bbR\log |F(-(iy)^2)|\frac {|\Re z|}{|z-iy|^2}dy\le
\sup_{y\in \bbR} \log |F(-(iy)^2)|.
$$
\end{proof}

\section{Functional Equation for the Extremal Function}

\begin{lemma} The extremal function $F$ has simple real zeros on $E$. Moreover,
\begin{equation}\label{6}
\int_E\frac{|F(x)|}{x-x_0} dx=0,
\end{equation}
for   an arbitrary zero $x_0$,  $F(x_{0})=0$.
\end{lemma}
\begin{proof}
As usual in such cases, see e.g. \cite{AKH67, YU11}, we can use the following variation of the extremal function
\begin{equation*}\label{5}
    F(z)\mapsto F(z)\left\{1+\frac{\epsilon z}{z-x_0}\right\}.
\end{equation*}
where  $F(x_0)=0$.  In particular \eqref{6} implies that 
\begin{equation*}\label{6bis}
\int_E\frac{|F(x)|}{(x-x_0)(x-x_1)} dx=0,
\end{equation*}
for $F(x_0)=F(x_1)=0$,
that is, an arbitrary gap $(a_j,b_j)$ does not contain more than one zero $x_0$. If such zero exists let us consider the linear function
$$
g(y_0)=\int_E (1-x y_0)\frac{|F(x)|}{1-x/x_0} dx=0,\quad 1/y_0\in(a_j,b_j).
$$
It assumes its maximum on the boundary of the domain, therefore the correction
\begin{equation*}\label{new5}
    F(z)\mapsto F(z)\frac{1-z y_0}{1-z/x_0}
\end{equation*}
improves the extremal function either for $y_0=1/a_j$ or for $y_0=1/b_j$.
\end{proof}

\begin{theorem}
Let $F$ be an extremal function, and let
\begin{equation}\label{7}
    w(z)=\int_E|xF(x)|\left\{\frac 1{x-z}-\frac 1 x\right\}dx
\end{equation}
and
\begin{equation}\label{8}
    S(z)=\int_E\sgn (xF(x))\left\{\frac 1{x-z}-\frac 1 x\right\}dx
\end{equation}
Then
\begin{equation}\label{9}
    w(z)=zF(z)(S(z)+M(\l)).
\end{equation}
\end{theorem}
\begin{proof}
Define
$$
H(z)=w(z)-zF(z)S(z).
$$
By \eqref{7} and \eqref{8} this is an entire function. Moreover, $H(0)=0$ and  by \eqref{6} $H(x_k)=0$ as soon as
$F(x_k)=0$. Thus $H(z)=z{F(z)}G(z)$, where $G(z)$ is an entire function.
Since
$$
G(z)=\frac{w(z)}{zF(z)}-S(z),
$$
we have
$$
\lim_{x\to\infty}\frac 1 x {G(x)}=0.
$$
Due to Lemma \ref{l21} $G(z)$ is  constant. Since $F(0)=1$, $w'(0)=M(\l)$ and $S(0)=0$ we get \eqref{9}.

\end{proof}

\section{The Extremal Problem and Canonical Systems}

In addition to \eqref{7}, \eqref{8} we define
\begin{equation}\label{ompm0}
    \om_+(z)=\int_{E_+}\left\{\frac 1{x-z}-\frac 1 x\right\}dx,\ \
    \om_-(z)=\int_{E_-}\left\{\frac 1{x-z}-\frac 1 x\right\}dx,
\end{equation}
and
\begin{equation}\label{em0}
    \e_-(z)=\frac 1 2\int_{\bbR\setminus E}(1-\sgn F(x))\left\{\frac 1{x-z}-\frac 1 x\right\}dx.
\end{equation}
where $E_\pm=\{x\in E: \sgn F(x)=\pm 1\}$.
In this case $S(z)=-\omega_+(z)+\omega_-(z)$.

We define the Nevanlinna class functions
\begin{equation}\label{R0}
\begin{split}
{R(z)}=&-\frac {1}{z}e^{\frac 1 2(\om_+(z) +\om_-(z))+\e_-(z)}\\
=&-\frac{\sqrt{z+1}}z\sqrt{\prod_{j\ge 1}\frac {1-z/a_j}{1-z/b_j}}
\prod_{j\ge 1}\left(\frac {1-z/b_j}{1-z/a_j}\right)^{\delta_j},
\end{split}
\end{equation}
where
$$
\delta_j=\begin{cases}1, &\sgn F(x)=-1, \ x\in (a_j,b_j)\\
0, &\sgn F(x)=1, \ x\in (a_j,b_j)
\end{cases}
$$
and
\begin{equation}\label{bd00}
\begin{split}
\frac A C
(z)=&-\frac {1}{z}e^{\frac 1 2 M(\l)+\om_-(z)+\e_-(z)}\\
=&-\frac { \prod_{j\ge 1}(1-z/\nu_j)}{z\lambda_0
\prod_{j\ge 1}(1-z/\lambda_j)},\ \lambda_0= e^{-\frac 1 2 M(\l)}.
\end{split}
\end{equation}
Note that $A(z)C(z)=-\lambda_0 zF(z)$.

The main functional equation \eqref{9} implies the following asymptotical equality
for the difference  of these two functions
\begin{equation}\label{asrac0}
\begin{split}
    R(z)-\frac{A(z)}{C(z)}=&\frac{A(z)}{C(z)}\left\{e^{\frac 1 2(\om_+(z) -\om_-(z)-M(\l))}-1\right\}\\
    \sim &\frac{A(z)}{C(z)}\frac 1 2(\om_+(z) -\om_-(z)-M(\l))=
    -\frac{A(z)}{C(z)}\frac 1 2(S(z)+M(\l))\\
    =&\frac{A(z)}{C(z)}\frac{w(z)}{-2zF(z)}=\frac{\lambda_0 w(z)}{2C^2(z)},\quad z\to+\infty.
    \end{split}
\end{equation}

If $A$ and $C$ are polynomials, this relation  means, that the fraction $A/C$ is the Pade approximation for $R$. As it well known the Pade approximation can be expressed by means of orthogonal polynomials with respect to the measure associated with the function $R$. Thus Peherstorfer's  Theorem 6 \cite{P88} follows.

Having in mind the fundamental Theorem \ref{thdB}   we prove
\begin{theorem}\label{th41}
There exists an entire $2\times 2$ matrix function
$$
\fA(z)= \begin{bmatrix}A&B\\C&D
     \end{bmatrix}(z), \quad \det\fA(z)=1,
     $$
     such that
     \begin{equation}\label{jexp}
 \frac{\fA^*(z)J\fA(z)-J}{z-\bar z}\ge 0,\quad J=\begin{bmatrix}
 0&-1\\ 1&0
 \end{bmatrix},
\end{equation}
     and a Nevanlinna class function $\cE(z)$ ($\Im \cE(z)\ge 0$, $\Im z>0$) such that
     \begin{equation}\label{repr}
        R(z)=\frac{A(z)\cE(z)+B(z)}{C(z)\cE(z)+D(z)}.
     \end{equation}

\end{theorem}

\begin{proof} If \eqref{repr} holds then
\begin{equation}\label{repr2}
     \cE(z)=\frac{D(z)R(z)-B(z)}{-C(z)R(z)+A(z)}.
\end{equation}
Since $AD-BC=1$, we have
\begin{equation}\label{rhox}
    \Im \cE(x)=\frac{\Im R}{|-C(x)R(x)+A(x)|^2}, \quad x\in\bbR.
\end{equation}
Now, let us define $\pi \rho(x)$, $x\in E$, by the RHS in \eqref{rhox}. We show that
\begin{equation}\label{rhoE}
    \int_E\rho(x)\frac{dx}{1+x^2}<\infty.
\end{equation}

By the  definition we have
\begin{equation*}
\begin{split}
\pi \rho(x)=&\frac{\Im R}{|-C(x)R(x)+A(x)|^2}=\frac{\Im (-\frac 1 xe^{\om_-+\e_-}e^{\frac 1 2(\om_+-\om_-)})}{|A|^2|RC/A-1|^2}\\
=&\frac{\Im (\frac A C e^{\frac 1 2(\om_+-\om_--M(\l))})}{|A|^2|e^{\frac 1 2(\om_+-\om_--M(\l))}-1|^2}.
\end{split}
\end{equation*}
Let us introduce the real valued functions
\begin{equation}\label{notow}
    \Re\om_\pm(x)=\alpha_\pm(x), \quad w(x)=u(x)+iv(x).
\end{equation}
Then, first of all, we have
$$
\frac{\Im R}{|-C(x)R(x)+A(x)|^2}=\frac{|A/C|e^{\frac 1 2(\alpha_+-\alpha_--M(\l))}}{|A|^2(e^{\alpha_+-\alpha_--M(\l)}+1)}=
\frac 1{2|AC|\cosh\frac{\alpha_+-\alpha_--M(\l)}2}.
$$
Secondarily, due to \eqref{9}
$$
u+iv=-\frac{AC}{\lambda_0}(\pm\pi i-\alpha_++\alpha_-+M(\l)).
$$
Therefore
$$
\frac{\sqrt{u^2+v^2}}v=\frac{\sqrt{\pi^2+(\alpha_+-\alpha_--M(\l))^2}}\pi
$$
and
$$
\frac {\lambda_0}{|AC|}=\frac
{\sqrt{\pi^2+(\alpha_+-\alpha_--M(\l))^2}}
{\sqrt{u^2+v^2}}=\frac{(\pi^2+(\alpha_+-\alpha_--M(\l))^2)v}{\pi(u^2+v^2)}.
$$
Thus
\begin{equation}
\frac{\Im R}{|-C(x)R(x)+A(x)|^2}= \frac{\pi^2+(\alpha_+-\alpha_--M(\l))^2}{2\lambda_0\cosh\frac{\alpha_+-\alpha_--M(\l)}2}\frac{v}{\pi(u^2+v^2)}.
\end{equation}
Since $-1/w$ is a Nevanlinna function the function ${v}/{\pi(u^2+v^2)}$ is integrable with respect to ${dx}/(1+x^2)$.

Now we discuss the point spectrum $\{y_j: y_j\in \bbR\setminus E\}$ related to $\cE(z)$. 
It has pole in the origin, moreover, since $A(z)C(z)=-z\lambda_0 F(z)$, and $A(0)=D(0)=F(0)=1$,
$$
\lim_{z\to 0}z\cE(z)=-\frac 1{1-\lambda_0}.
$$
 Otherwise, since $AC(y_j)\not=0$, we have  $(CR-A)(y_j)=0$.
We evaluate the corresponding mass
$$
\rho_j=\lim_{x\to y_j}(y_j-x)\cE(x).
$$
Using $AD-BC=1$ and \eqref{repr2}, we have
\begin{equation*}\label{rhodef2}
    \rho_j=\frac 1{C(y_j)(CR-A)'(y_j)}.   
\end{equation*}
Since
$$
CR-A=A(e^{\frac 1 2(\om_+-\om_--M(\l))}-1),
$$
we get
$(\om_+-\om_--M(\l))(y_j)=0$, that is, 
\begin{equation*}
\begin{split}
C(y_j)(CR-A)'(y_j)  =AC(y_j)(e^{\frac 1 2(\om_+-\om_--M(\l))}-1)'(y_j)\\
=
\frac 1 2 AC(y_j)(\om_+-\om_--M(\l))'(y_j).
\end{split}
\end{equation*}
On the other hand, by \eqref{9},
$
\lambda_0w=AC(\om_+-\om_--M(\l)).
$
Therefore
\begin{equation*}\label{rh3}
    \lambda_0 w'(y_j)=(AC)(y_j)(\om_+-\om_--M(\l))'(y_j),
\end{equation*}
and $\rho_j=\frac 2{\lambda_0 w'(y_j)}$.

Again, since $-1/w$ is a Nevanlinna function and $w(y_j)=0$ we get 
\begin{equation}\label{goal2}
    \rho_j>0\quad\text{and}\quad \sum \frac{\rho_j}{1+y_j^2}<\infty.
\end{equation}

Having \eqref{rhoE} and \eqref{goal2} for given $R, A, C$ and $w$ we define $\cE$ by the integral representation
\begin{equation}\label{intrepr}
\begin{split}
    \cE(z)=\frac{D(z)R(z)-B(z)}{-C(z)R(z)+A(z)}:=&-\frac 1 {(1-\lambda_0)z} +\tilde q_0+
    \int_E \left\{\frac 1{x-z}-\frac 1 x\right\}\rho(x)\,dx\\+&
    \sum\left\{\frac 1{y_j-z}-\frac 1{y_j}\right\}\rho_j,  
    \end{split}
\end{equation}
where $\tilde q_0$ is uniquely defined by the normalization in the origin $D(0)=1$, $B(0)=0$.

The fact that $\fA$ is $J$-expanding, that is  \eqref{jexp}, is proved in Lemma \ref{l45} below.
\end{proof}

Several next claims on the way to prove Lemma \ref{l45} are important on their own and we formulate them as separate Lemmas.

Note that  $\cE(z)$ is a function on the two-sheeted Riemann surface related to $R(z)$. Next Lemma describes its property in the extension on the lower sheet, i.e., for $R\to-R$.
In fact, the matrix
$$
\begin{bmatrix}
D(z)& B(z)\\C(z)&A(z)
\end{bmatrix}
$$
is $J$--inner simultaneously with the matrix $\fA$, therefore the function
\begin{equation}\label{cem}
    \cE_-(z)=\frac{D(z)R(z)+B(z)}{C(z)R(z)+A(z)}
\end{equation}
should be also of the Nevanlinna class. 

\begin{lemma}\label{le42} If $\cE$ is defined by \eqref{intrepr} then 
$\cE_-(z)$, defined by \eqref{cem},
is also of Nevanlinna class.
\end{lemma}

\begin{proof}
Note that due to $\det \fA(z)=1$
\begin{equation}\label{cetwo}
\cE(z)=\frac{R(z)}{-C^2(z)R^2(z)+A^2(z)}+\frac{D(z)C(z)R^2(z)-A(z) B(z)}{-C^2(z)R^2(z)+A^2(z)}
\end{equation}
and consider the difference
\begin{equation}\label{intrepr2}
\begin{split}
 \frac{R(z)}{-C^2(z)R^2(z)+A^2(z)}
-&\left\{-\frac 1 {(1-\lambda^2_0)z} +
    \int_E \left\{\frac 1{x-z}-\frac 1 x\right\}\rho(x)\,dx\right.\\ +&
   \left. \frac 1 2 \sum\left\{\frac 1{y_j-z}-\frac 1{y_j}\right\}\rho_j\right\}. 
    \end{split}
\end{equation}
Both functions have the same imaginary part on $E$. For zeros $y_j$ of $A-CR$ in gaps we have
$$
\frac{R(y_j)}{(CR+A)(y_j)(CR-A)'(y_j)}= \frac 1 {2C(y_j)(CR-A)'(y_j)}=\frac 1 2\rho_j.
$$
Finally, $(CR+A)(x)\not=0$ in gaps since $e^{\frac 1 2(\omega_+-\omega_--M(\l))(x)}+1>0$.
Thus the difference in \eqref{intrepr2} is an entire function of bounded characteristic in $\bbC\setminus E$. To use Lemma \ref{l21} we have to estimate its grow on the positive half axis. By \eqref{asrac0}
$$
C^2(z)\left(R(z)-\frac{A(z)}{C(z)}\right)\frac{R(z)+\frac{A(z)}{C(z)}}{R(z)}\sim\lambda_0 w(z).
$$
Therefore the grow  is of the form $o(x)$, that is,
this entire function is  a constant, which we denote by $\tilde q_0^{(1)}$. 

As the result we get an integral representation for the first term in \eqref{cetwo}, as well as for the second one
$$
\frac{D(z)C(z)R^2(z)-A(z) B(z)}{-C^2(z)R^2(z)+A^2(z)}=
-\frac {\lambda_0} {(1-\lambda^2_0)z} +\tilde q_0^{(2)}+
 \sum\left\{\frac 1{y_j-z}-\frac 1{y_j}\right\}\frac{\rho_j} 2,
$$
where $\tilde q_0=\tilde q_0^{(1)}+\tilde q_0^{(2)}$. Thus
\begin{equation}\label{ceminu}
\begin{split}
\cE_-(z)=&\frac{R(z)}{-C^2(z)R^2(z)+A^2(z)}-\frac{D(z)C(z)R^2(z)-A(z) B(z)}{-C^2(z)R^2(z)+A^2(z)}
\\
=&-\frac 1 {(1+\lambda_0)z} +\tilde q_0^{(1)}-\tilde q_0^{(2)}+
    \int_E \left\{\frac 1{x-z}-\frac 1 x\right\}\rho(x)\,dx.
    \end{split}
\end{equation}
\end{proof}

\begin{lemma}
For  the given $R$, $A$, and $C$
\begin{equation}\label{intac0}
\int_E C^2(x) \frac{|R(x)|dx}{\pi (1+x^2)}<\infty, \quad \int_E A^2(x) \frac{dx}{\pi  |R(x)| (1+ x^2)}<\infty.
\end{equation}
Moreover
\begin{equation}\label{intac}
R(z) C^2(z)-A(z) C(z)=\int_E
\left\{\frac 1{x-z}-\frac 1 x\right\}
C^2(x) \frac{|R(x)|dx}{\pi}
\end{equation}
and
\begin{equation}\label{intac2}
A(z) C(z)-\frac{A^2(z)}{R(z)}=\int_E
\left\{\frac 1{x-z}-\frac 1 x\right\}
A^2(x) \frac{dx}{\pi |R(x)|}.
\end{equation}
\end{lemma}

\begin{proof} We proved that  $\cE$ and $\cE_-$ are of Nevanlinna class.
Therefore
\begin{equation}\label{W}
   -\left\{\frac{\cE(z)+\cE_-(z)} 2\right\}^{-1}=\frac{C^2(z)R^2(z)-A^2(z)}{R(z)}.
\end{equation}
is also of this class.  Thus, its imaginary part is integrable and we get \eqref{intac0}. Hence, the integrals in \eqref{intac} and \eqref{intac2} are well defined.
Again, using Lemma \ref{l21} and the asymptotics \eqref{asrac0} we get that the difference between LHS and RHS in \eqref{intac} is a constant function. Since both functions vanish in the origin, we have \eqref{intac}. The same arguments prove \eqref{intac2}.
\end{proof}

\begin{corollary}
\begin{equation}\label{4main}
e^{\frac 1 2 M(\l)}=1+\int_E\left(\frac{C(x)}{C'(0) x}\right)^2\frac{|R(x)|dx}{\pi}.
\end{equation}
\begin{proof}
We divide both parts of \eqref{intac} by $z$ and then substitute $0$ for $z$. Recall that $C'(0)=-\lambda_0=
-e^{-\frac 1 2 M(\l)}$.
\end{proof}
\end{corollary}

%%%%%%%%%%%

\begin{lemma}\label{l45}
Let $A$ and $C$ be defined by \eqref{bd00} and $B$ and $D$ by  \eqref{intrepr}. These entire  functions form $J$-expanding matrix, that is, \eqref{jexp} holds.
\end{lemma}

\begin{proof}
\eqref{jexp} is equivalent to
$$
\frac{\Phi(z)^*\Psi(z)-\Psi(z)^*\Phi(z)
}{z-\bar z}\ge 0,
$$
where
$$
\Phi(z)=
\begin{bmatrix}0&0\\ 0&1
\end{bmatrix}\fA(z)+
\begin{bmatrix}1&0\\0&0
\end{bmatrix}=
\begin{bmatrix}1&0\\ C(z)&D(z)
\end{bmatrix}
$$
and
$$
\Psi(z)=
\begin{bmatrix}0&0\\1&0
\end{bmatrix}\fA(z)+
\begin{bmatrix}0&1\\0&0
\end{bmatrix}
=
\begin{bmatrix}0&1\\ A(z)&B(z)
\end{bmatrix},
$$
that is, to the fact that the $2\times 2$ matrix function
$$
W(z):=-\Phi(z)\Psi(z)^{-1}=
-\begin{bmatrix}1&0\\ C(z)&D(z)
\end{bmatrix}\begin{bmatrix}0&1\\ A(z)&B(z)
\end{bmatrix}^{-1}.
$$
has positive imaginary part. Using $AB-CD=1$ we have
$$
W(z)=\begin{bmatrix}1&0\\ C(z)&D(z)
\end{bmatrix}\begin{bmatrix}(B/A)(z)&-(1/A)(z)\\ -1&0
\end{bmatrix}=
\frac{
\begin{bmatrix}-B(z)&1\\ 1&C(z)
\end{bmatrix}}{-A(z)}.
$$

Note that $-C/A$ belongs to the Nevanlinna class. We have to show that $B/A$  has also  positive imaginary part in the upper half plane. We use
$$
\cE(z)=\frac{R(z)}{A(z)(-C(z)R(z)+A(z))}-\frac{B(z)}{A(z)}.
$$
By \eqref{intac2} the function $AC-A^2/R$ belongs to the Nevanlionna class, as well as the reciprocal $R/(-AC R+A^2)$.  Thus $B/A$ groves not faster than a linear function on the positive half axis, moreover
$$
\lim_{z\to+\infty}\frac{B(z)}{zA(z)}=
\lim_{z\to+\infty}\frac{R(z)}{zA(z)(-C(z)R(z)+A(z))}=\sigma_1>0.
$$

By Lemma \ref{l21} and $B(\nu_j)C(\nu_j)=1$, the function $B/A$ is an entry in the following matrix representation
$$
W(z)=
\begin{bmatrix}\sigma_1&0\\ 0&\sigma_2
\end{bmatrix}z+
\begin{bmatrix}0&1\\ 1&0
\end{bmatrix}+
\sum_{j\ge 1}\frac{\begin{bmatrix}(1/C)(\nu_j)&1\\ 1&C(\nu_j)
\end{bmatrix}}{-A'(\nu_j)}\left\{\frac{1}{z-\nu_j}-\frac{1}{\nu_j}\right\}.
$$
Since the matrices 
$$
\frac{\begin{bmatrix}(1/C)(\nu_j)&1\\ 1&C(\nu_j)
\end{bmatrix}}{-A'(\nu_j)}\ge 0,\ j\ge 1, \quad \begin{bmatrix}\sigma_1&0\\ 0&\sigma_2
\end{bmatrix}\ge 0
$$
are nonegative, we obtain $(W(z)-W(z)^*)/(z-\bar z)\ge 0$, and the lemma is proved.
\end{proof}

\section{In the presence of DCT}
We say that a function $F$ of bounded characteristic belongs to the Smirnov class  if its  singular component is bounded. 

Let $\Omega=\bbC\setminus E$ be of Widom type. The Direct Cauchy Theorem (DCT) holds in
$\Omega$ if
\begin{equation}\label{eqdct1}
\frac 1{2\pi i}\int_E F(x+i0)\frac{dx} x-\frac 1{2\pi i}\int_E F(x-i0) \frac{dx}{x}=F(0)
\end{equation}
for every function $F$ of Smirnov class such that 
\begin{equation}\label{eqdct2}
\int_E |F(x+i0)|\frac{dx}{|x|}+\int_E |F(x-i0)| \frac{dx}{|x|}<\infty.
\end{equation}

If $\bbC\setminus E\simeq \bbD/\Gamma$ is of Widom type, then there exists a measurable fundamental set $\bbE$ (with respect to the Lebesgue measure $dm$) for the action of the group $\Gamma$ on $\bbT$ \cite{POM}, i.e.,
\begin{equation*}
\begin{split}
1)&\bbE\cap\gamma(\bbE)=\emptyset, \quad \gamma\not=I\\
2)&m(\cup_{\gamma_\in\Gamma}\gamma(\bbE))=m(\bbT).
\end{split}
\end{equation*}
For an analytic function $f$ in $\bbD$ we write
$$
f|[\gamma](\zeta)=\frac{f(\gamma(\zeta))}{\gamma_{21}\zeta+\gamma_{22}}, \quad\gamma=
\begin{bmatrix}
\gamma_{11}&\gamma_{12}\\
\gamma_{21}&\gamma_{22}
\end{bmatrix}.
$$

For $\alpha\in\Gamma^*$, the space $A_1^2(\alpha)$ is formed by Smirnov class functions $f$ in $\bbD$ such that
$$
f|[\gamma]=\alpha f,\ \forall\gamma\in\Gamma,\quad
\|f\|^2=\int_{\bbE}|f|^2dm<\infty.
$$
The point evaluation functional is bounded in this space and we denote by $k^\alpha(\zeta)$ the reproducing kernel in the origin
$$
\langle f, k^\alpha\rangle=f(0), \quad f\in A_1^2(\alpha).
$$
By $b$ we denote the Green function of the group $\Gamma$ with respect to $\zeta_0$
$$
b_{\zeta_0}(\zeta)=\prod_{\gamma\in\Gamma}\frac{|\gamma(\zeta_0)|}{\gamma(\zeta_0)}\frac{\gamma(\zeta_0)-\zeta}{1-\zeta\gamma(\zeta_0)}. 
$$
Note that $b_{\zeta_0}$ is character-automorphic $b_{\zeta_0}\circ\gamma=\mu_{\zeta_0}(\gamma)b_{\zeta_0}$ and it is related to the standard Green function in the domain by
$$
\ln\frac 1{|b_{\zeta_0}(\zeta)|}=G(z(\zeta),z(\zeta_0)).
$$ 
By $b$ we denote the Green function with respect to the origin, $b=b_0$.

\begin{lemma}
DCT is equivalent to  each of the following to identities
\begin{equation}\label{eqn61}
b(\zeta)\overline{k^\alpha(\zeta)}=\zeta\frac{k^{\mu\alpha^{-1}}(\zeta)}{k^{\mu\alpha^{-1}}(0)}b'(0),
\ \zeta\in\bbT, \quad \forall \alpha\in\Gamma^*,
\end{equation}
and 
\begin{equation}\label{eqn62}
{k^\alpha(0)}{k^{\mu\alpha^{-1}}(0)}=b'(0)^2,  \quad \forall \alpha\in\Gamma^*.
\end{equation}
\end{lemma}

Our further construction is based on the following theorem \cite{YU}. Recall that
the character $\xi_t$ was relaeted to the complex Martin function $\Theta$ in \eqref{eq23}.

\begin{theorem}\label{th52} Let $\bbC\setminus E$ be of Widom type with DCT. Let $\fB(t,z)$ corresponds to the canonical system with the resolvent function
\begin{equation}\label{eqn63}
r(z)=r(z,\{\delta_j\}_{j\ge 1})=\frac 1 2\left\{
-\frac{\sqrt{1+z}}z \sqrt{\prod_{j\ge 1}\left(\frac{1-z/a_j}{1-z/b_j}\right)^{\delta_j}}
+\frac 1 z-q_0\right\},
\end{equation}
where $\delta_j=\pm 1$ and $q_0=q_0(\{\delta_j\}_{j\ge 1})$ is defined by the condition $r(0)=0$.
For the given $\{\delta_j\}_{j\ge 1}$ there exists a unique square root $\nu$ of  $\mu$,  
$\nu^2=\mu$, such that
$$
r(z(\zeta))=\frac{b'(0)}{z'(0)}\frac{b(\zeta)k^{\nu^{-1}}(\zeta)}{k^\nu(\zeta)}
$$ 
Moreover,
\begin{equation*}
\begin{bmatrix}
\frac{b'(0)}{z'(0)}{b(\zeta)k^{\nu^{-1}}(\zeta)}\\{k^\nu(\zeta)}
\end{bmatrix}\frac{e^{-it\Theta(z(\zeta))}}{\sqrt{k^{\nu}(0)}}=\fB(t,z(\zeta))
\begin{bmatrix}
\frac{b'(0)}{z'(0)}{b(\zeta)k^{\nu^{-1}\xi_t^{-1}}(\zeta)}\\{k^{\nu\xi_t^{-1}}(\zeta)}
\end{bmatrix}\frac 1 {\sqrt{k^{\nu\xi_t^{-1}}(0)}},
\end{equation*}
where $\fB$ meets the normalization
\begin{equation}\label{eqn64}
\fB(t,0)=\begin{bmatrix}\tau& 0\\ 0&\tau^{-1}\end{bmatrix},\ 
\begin{bmatrix}0&1\end{bmatrix}\fB'(t,0)\begin{bmatrix}
1\\ 0\end{bmatrix}
=\tau-\tau^{-1},\ \ 
\tau^2=\frac{e^{-2t}k^{\nu\xi_t^{-1}}(0)}{k^{\nu}(0)}.
\end{equation}

\end{theorem}

We want to pass from the canonical system related to the spectral function $r(z)$ given by \eqref{eqn63} to the spectral function $R(z)$ of the family \eqref{rth01}. The following two transformations are quite evident.
\begin{lemma}\label{l63}
If  $R\mapsto \lambda R$, $\lambda>0$, then 
$$
\begin{bmatrix}
A& B\\
C&D
\end{bmatrix}\mapsto
\begin{bmatrix}
A& \lambda B\\
\frac 1 \lambda C&D
\end{bmatrix}.
$$
If $R\to R+q$, $q\in \bbR$, then
$$
\begin{bmatrix}
A& B\\
C&D
\end{bmatrix}\mapsto
\begin{bmatrix}
1& q\\
0&1
\end{bmatrix}
\begin{bmatrix}
A& B\\
C&D
\end{bmatrix}
\begin{bmatrix}
1&-q \\
0&1
\end{bmatrix}.
$$
\end{lemma}

A little bit more involved transformation is given in the following lemma.

\begin{lemma}\label{l64}
Let $\fB(z)$, $\fB(0)=I$, and $\fA(z)$, $\fA(0)=I$, correspond to the systems with the resolvent functions $R(z)$ and 
$R(z)-\frac 1 z$ correspondently. If
$$
\fB(z)\begin{bmatrix} 1\\ 0\end{bmatrix}=\begin{bmatrix} A(z)\\ C(z)\end{bmatrix}
$$
then
$$
\fA(z)\begin{bmatrix} 1\\ 0\end{bmatrix}=\begin{bmatrix} A(z)-\frac 1 z C(z)\\ C(z)\end{bmatrix}
\frac 1{1-C'(0)}.
$$
\end{lemma}

\begin{proof}
First we note that $\fA(z)$ is of the form
$$
\fA(z)=\begin{bmatrix} 1 &-\frac 1 z\\
0&1
\end{bmatrix}\fB(z)\begin{bmatrix} 1 &\frac \rho z\\
0&1
\end{bmatrix} U
$$
where $\rho$ is fixed by the condition that the resulting matrix function $\fA(z)$ does not have pole in the origin, that is,
\begin{equation}\label{eq61}
-\begin{bmatrix} 0 &1\\
0&0
\end{bmatrix}+
\rho\begin{bmatrix} 0 &1\\
0&0
\end{bmatrix}-
\rho\begin{bmatrix} 0 &1\\
0&0
\end{bmatrix}\fB'(0)\begin{bmatrix} 0 &1\\
0&0
\end{bmatrix}=0;
\end{equation}
and $U$ is a $J$-unitary constant matrix, which is defined by the condition $\fA(0)=I$.

\eqref{eq61} is equivalent to
$$
-1+\rho(1-C'(0))=0.
$$
For the given $\rho$, $U=\tilde \fA^{-1}(0)$, where
$$
\tilde\fA(z)=\begin{bmatrix} 1 &-\frac 1 z\\
0&1
\end{bmatrix}\fB(z)\begin{bmatrix} 1 &\frac \rho z\\
0&1
\end{bmatrix}.
$$
We have
$$
\tilde\fA(0)=I-\begin{bmatrix} 0 &1\\
0&0
\end{bmatrix}\fB'(0)+\rho\fB'(0)
\begin{bmatrix} 0 &1\\
0&0
\end{bmatrix}-\frac\rho 2\begin{bmatrix} 0 &1\\
0&0
\end{bmatrix}\fB''(0)\begin{bmatrix} 0 &1\\
0&0
\end{bmatrix}.
$$
This matrix is $J$-unitary, therefore
\begin{equation*}
\begin{split}
\fA^{-1}(0)=&J\fA(0)^*J^*\\
=&
I-J\fB'(0)^*\begin{bmatrix} 0 &0\\
0&1
\end{bmatrix}-\rho
\begin{bmatrix} 1 &0\\
0&0
\end{bmatrix}
\fB'(0)^*J^*
+\frac\rho 2\begin{bmatrix} 1 &0\\
0&0
\end{bmatrix}\fB''(0)^*\begin{bmatrix} 0 &0\\
0&1
\end{bmatrix}.
\end{split}
\end{equation*}
In particular
$$
\fA^{-1}(0)\begin{bmatrix} 1 \\0
\end{bmatrix}=
\begin{bmatrix} 1 \\0
\end{bmatrix}+\rho\begin{bmatrix} 1 \\0
\end{bmatrix}C'(0)=\frac 1{1-C'(0)}\begin{bmatrix} 1 \\0
\end{bmatrix}.
$$

Thus
\begin{equation*}
\begin{split}
\fA(z)\begin{bmatrix} 1 \\0
\end{bmatrix}=&\tilde \fA(z)\tilde\fA^{-1}(0)\begin{bmatrix} 1 \\0
\end{bmatrix}=\begin{bmatrix} 1 &-\frac 1 z\\
0&1
\end{bmatrix}\frac{\fB(z)}{1-C'(0)}\begin{bmatrix} 1 \\0
\end{bmatrix}\\=&
\begin{bmatrix} A(z)-\frac 1 z C(z)\\ C(z)\end{bmatrix}
\frac 1{1-C'(0)}.
\end{split}
\end{equation*}

\end{proof}

We are in a position to  prove one  of our main results.
\begin{theorem}\label{th55}
Let DCT hold in the Widom domain $\bbC\setminus E$. Then the infimum $M(2\l)$ in \eqref{1} is given by the expression
\begin{equation}\label{eqfinal}
M(2\l)=\inf_{\nu:\nu^2=\mu}2\ln\frac{k^{\nu}(0)+e^{-2\l}k^{\nu\xi_\l^{-1}}(0)}{k^{\nu}(0)-e^{-2\l}k^{\nu\xi_\l^{-1}}(0)}.
\end{equation}

\end{theorem}

\begin{proof}
We start with $\fB(z)$ related to the spectral function \eqref{eqn63}. Let
$$
\fA_1(z)=\begin{bmatrix} A_1&B_1\\ C_1& D_1\end{bmatrix}(z)=
\fB(z)\fB^{-1}(0)
$$
By \eqref{eqn64}, $C_1'(0)=1-\tau^{-2}$. We apply Lemma \ref{l63} and multiply this function by 2 and shift by $q_0$. As the result we get  $\fA_2(z)$ such that $C'_2(0)=(1-\tau^{-2})/2$. Then we apply Lemma \ref{l64} to get $\fA$ related to the spectral function $R(z)$ of the family \eqref{1} such that
$$
C'(0)=C_R'(0)=\frac{C'_2(0)}{1-C_2'(0)}=\frac{1-\tau^{-2}}{1+\tau^{-2}}=
-\frac{1-\tau^{2}}{1+\tau^{2}}.
$$
Making of  use  \eqref{eqn64} we obtain
$$
C_R'(0)=C'_\nu(t,0)=-\frac{k^{\nu}(0)-e^{-2t}k^{\nu\xi_t^{-1}}(0)}{k^{\nu}(0)+e^{-2t}k^{\nu\xi_t^{-1}}(0)}.
$$
By Theorem \eqref{th41} the extremal function $F(z)$ of the Problem \ref{pr1} is of the form \eqref{sol0} with a certain $\{\delta_j\}_{j\ge1}$ and $t_0$. Therefore the supremum over all $\nu$ for the fixed exponential type $t_0=\l$ gives 
$$
\sup_{R\in\cR}(-C_R'(\l,0))=\lambda_0=e^{-\frac 1 2 M(2\l)},
$$
or, equivalently, \eqref{eqfinal}.
\end{proof}

\begin{remark}
Note the highly important  property $k^\nu(0)=k^{\nu_0}(0)$ for a fixed $\nu_0^2=\mu$. Therefore we can define the function
\begin{equation}\label{pidef}
\Pi(\alpha)=\inf_{j^2=1_{\Gamma^*}} \frac{k^{\nu_0 j\alpha}(0)}{k^{\nu_0}(0)},
\end{equation}
which is evidently of periods $j$'s, $j^2=1_{\Gamma^*}$. In these notations
\begin{equation}\label{pimain}
M(2\l)= 2\ln\frac{1+e^{-2\l}\Pi(\xi_\l^{-1})}{1-e^{-2\l}\Pi(\xi_\l^{-1})}.
\end{equation}

\end{remark}

\begin{remark}
In  the simplest case $E=(-\infty,-1]$.
Since the group $\Gamma^*$ is trivial $k^\alpha$, in fact, does not depend on $\alpha$ and we get from \eqref{eqfinal}
$$
M(2\l)=2\ln \coth \l.
$$
Note that $M(2\l)$ is the error of approximation of $1/x$ on the half-axis $(-\infty,-1]$ by entire functions $G$ of exponential type $2\l$ of the order $1/2$ in $L^1$ norm, i.e.,
$$
M(2\l)=\inf_{G\in \cB_{2l}}\int^{-1}_{-\infty}\left |\frac 1 x -G(x)
\right| dx,
$$
and the extremal function is of the form
$$
G_{2\l}(z)=\frac 1 z\left(1-\frac{\sinh2\sqrt{z+1}\l}{\sqrt{z+1}2\l}\right).
$$
We did not find a solution in the literature, but of course it can be obtain as a limit from the polynomial case \cite[Appendix, 45]{AKH61}
$$
M_n=\inf_{P\in\cP_{n-1}}
\int^{1}_{-1}\left |\frac 1 {x-u} -P(x)
\right| dx=2\ln\frac{1+v^n}{1-v^n}, \ u=\frac {v+v^{-1}}2, \ v\in(0,1).
$$
\end{remark}

\section{If DCT fails}
\begin{theorem}\label{th61}
Let $E$ be a system of intervals \eqref{eq101.5} such that $\Omega=\bbC\setminus E$ is of Widom type. DCT fails in $\Omega$ if and only if $M(0)<\infty$. Moreover, there exists a single valued function $I(z(\zeta))=\prod b_{\zeta_k}(\zeta)$, $z(\zeta_k)\in (a_k,b_k)$ such that
\begin{equation}\label{eqno61}
M(0)=2\ln\frac{1+I(0)}{1-I(0)}
\end{equation}
and an extremal function $F\in\cB_0(E)$ is of the form
\begin{equation}\label{eqno62}
F(z)=\frac 1{\sqrt{z+1}}\prod_{j\ge 1}\frac{1-z/z(\zeta_j)}{\sqrt{(1-z/a_j)(1-z/b_j)}}\frac{I^{-1}(z)-I(z)}{I^{-1}(0)-I(0)}.
\end{equation}

\end{theorem}

\begin{proof}
A complete solution of the extremal problem in the presence of DCT is given in the previous section. In particular $M(0)=\infty$ in this case.

Assume that $M(0)<\infty$ and let $F\in B_0$, $F(0)=1$, be such that the integral \eqref{eqdct2} is finite. It means that $F$ is of Smirnov class,  since the only possible unbounded  inner part of this function is of the form $e^{-i\l\Theta}$. Thus, due to DCT the integral on the boundary of the domain $\Omega$ in \eqref{eqdct1} is 1. On the other hand
$F(x+i0)=F(x-i0)$, that is the integral vanishes.

Let, now, $R\in \cR$ correspond to an extremal configuration of ends of the intervals and 
$F=-AC/(\lambda_0z)$ be the extremal function, $\lambda_0=e^{-\frac 1 2 M(0)}$. Consider the Nevanlinna class function given by \eqref{W}. Taking into account that this function is real in $\bbR\setminus E$ and takes imaginary values in $E$, we have the following representation
\begin{equation}\label{eqno63}
\frac{C^2(z) R^2(z)-A^2(z)}{R(z)}=\frac{(1-\lambda_0^2)z}{\sqrt{1+z}}\prod_{j\ge1}
\frac{1-z/z_j}{\sqrt{(1-z/a_j)(1-z/b_j)}}
\end{equation}
with certain $z_k\in[a_k,b_k]$.  Define
\begin{equation}\label{eqno64}
\begin{split}
\Phi=&A+CR\\
I\Phi =&A-CR
\end{split}
\end{equation}
As it was discussed, see the proof of Lemma \ref{le42}, the first function has no zeros. Since $A$ and $C$ belong to $\cB_0$, it does not have the singular inner factor related to infinity. That is $\Phi$ is an outer function. Thus $I$ is the inner factor of $A-CR$, what is the same, of $A^2-C^2R^2$. Due to \eqref{eqno63} $I(z)=\prod b_{\zeta_k}$, where $z_k=z(\zeta_k)$. Let us point out that $I$ is single valued in $\Omega$,  as well as $\Phi$.

In these notations we have
\begin{equation}\label{eqno65}
\begin{split}
A(z)C(z)=&\frac 1 4 (I^{-1}(z)-I(z))\frac{I(z)\Phi^2(z)}{R(z)}\\
=&\frac 1 4 (I^{-1}(z)-I(z))\frac{-(1-\lambda_0^2)z}{\sqrt{1+z}}\prod_{j\ge1}
\frac{1-z/z_j}{\sqrt{(1-z/a_j)(1-z/b_j)}}.
\end{split}
\end{equation}
Due to the normalization in the origin we have
$$
4\lambda_0=(I^{-1}(0)-I(0))(1-\lambda_0^2).
$$
It proves \eqref{eqno61} and, together with \eqref{eqno65}, \eqref{eqno62}.
\end{proof}

We see that DCT plays another important role as a characteristic property of the canonical systems with reflectionless spectral functions.

Recall that a Nevanlinna class function $\cE(z)$  is a reflectionless spectral function on $E$ if there exists  a Nevanlinna class function $\cE_-(z)$ such that $\overline{\cE(x)}=-\cE_-(x)$ for $x\in E$ and $-(\cE(z)+\cE_-(z))^{-1}$  is holomorphic in $\Omega$. We assume that each function of the class is normalized by the condition 
$$
\cE_-(z)=-\frac 1 z+\dots
$$
in the origin.

\begin{theorem}\label{th62}
For a reflectionless function $\cE(z)$ let $\fA(t,z)$  be related to the corresponding canonical system. $t$ is in one-to-one correspondence with the exponential type $\l$, given by
$$
\l(t)=\lim_{x\to+\infty}\frac{\|\fA(t,x)\|}{\cM(x)},
$$ 
for all $\cE$ of the class if and only if DCT holds.
\end{theorem}

\begin{proof}
If DCT holds all reflectionless canonical systems were parametrized in \cite{YU}, see Theorem 
\ref{th52}. If DCT fails we use the extremal function $R$, described in the previous theorem, to define the normalized reflectionless spectral function $\cE(z)=\frac 1 2\left(R(z)+\frac 1 z\right)$.
Then, for a certrain $t_0>0$,
$$
\fA(t_0,z)\begin{bmatrix}1\\ 0\end{bmatrix}=\begin{bmatrix}A(z)+\frac 1 z C(z)\\
2 C(z)
\end{bmatrix}\frac 1{1+C'(0)}
$$
and $\l(t_0)=0$. Moreover $\l(t)=0$ for all $t\in[0, t_0]$. Thus, for the given $\cE(z)$ the function $\l(t)$ is not one-to-one.
\end{proof}

\bibliographystyle{amsplain}

 \vspace{.1in}

Abteilung f\"ur Dynamische Systeme und Approximationstheorie,

Johannes Kepler Universit\"at Linz,

A--4040 Linz, Austria

Petro.Yudytskiy@jku.at

\end{document}